\documentclass[11pt]{article}

\usepackage[T1]{fontenc}
\usepackage{lmodern}
\usepackage{amsmath,amssymb,amsthm,mathtools}
\usepackage{microtype}
\usepackage[margin=1in]{geometry}
\usepackage[hidelinks]{hyperref}

\numberwithin{equation}{section}

\newtheorem{theorem}{Theorem}[section]
\newtheorem{proposition}[theorem]{Proposition}
\newtheorem{lemma}[theorem]{Lemma}
\newtheorem{corollary}[theorem]{Corollary}
\newtheorem*{question}{Question}
\theoremstyle{remark}
\newtheorem{remark}[theorem]{Remark}

\newcommand{\N}{\mathbb{N}}
\newcommand{\Pplus}{P^{+}}
\newcommand{\e}{\mathrm{e}}
\DeclareMathOperator{\lcm}{lcm}

\title{Long Intervals Without Distinct Multiples\\
of the First \(n\) Positive Integers}
\author{Scott Duke Kominers%
\thanks{\textsc{Harvard Business School; Department of Economics and Center of
Mathematical Sciences and Applications, Harvard University; and a16z crypto.}}}
\date{July 11, 2026}

\begin{document}

\maketitle

\begin{abstract}
For positive integers \(n\) and \(m\), let \(f(n,m)\) be the least
integer \(h\ge0\) such that \((m,m+h]\) contains distinct
integers \(a_1,\ldots,a_n\) satisfying \(i\mid a_i\) for
\(1\le i\le n\), and put \(F(n)=\max_{m\in\N} f(n,m)\).
A recent theorem of van Doorn gives
\(F(n)-f(n,n)>0.36\,n\log n/\log\log n\) for sufficiently large~\(n\).
We prove
\[
  \liminf_{n\to\infty}
  \frac{F(n)-f(n,n)}{n\log n}
  \ge \frac{1}{\e}.
\]
Thus, for every fixed \(c<1/\e\) and all sufficiently large~\(n\),
some interval of length \(c\,n\log n\) contains no system of
pairwise distinct multiples of \(1,2,\ldots,n\).
The proof applies an Erd\H{o}s--Pomerance smooth-number obstruction at
starting points \(m\asymp n\log n\), using local saddle-point estimates
of Hildebrand and Tenenbaum.
\end{abstract}

\section{Introduction}

Throughout the paper \(\N=\{1,2,3,\ldots\}\), and all logarithms are
natural.  For \(n,m\in\N\), define \(f(n,m)\) to be the least
nonnegative integer \(h\) such that the interval \((m,m+h]\) contains
distinct integers \(a_1,\ldots,a_n\) with
\[
  i\mid a_i \qquad (1\le i\le n).
\]
Every value \(f(n,m)\) is finite; indeed \(f(n,m)\le n(n+1)/2\), as one
sees by choosing \(a_n,a_{n-1},\ldots,a_1\) in this order, taking each
\(a_i\) to be the least multiple of \(i\) exceeding the previously
chosen element (and exceeding \(m\) when \(i=n\)).  Moreover, for fixed
\(n\) the function \(m\mapsto f(n,m)\) is periodic modulo
\(\lcm(1,2,\ldots,n)\), since translation by \(\lcm(1,2,\ldots,n)\) is a bijection
between the admissible systems of multiples in \((m,m+h]\) and those
in the translated interval.  Hence
\[
  F(n):=\max_{m\in\N} f(n,m)
\]
is finite and is attained.

The study of \(f(n,m)\) was initiated by Erd\H{o}s and Pomerance
\cite{ErdosPomerance}, who determined the diagonal quantity \(f(n,n)\)
up to a factor \(\sqrt{\log\log n}\),
\begin{equation}\label{eq:EP-bounds}
  \left(\frac{2}{\sqrt{\e}}+o(1)\right)
  n\sqrt{\frac{\log n}{\log\log n}}
  \;<\;
  f(n,n)
  \;<\;
  \bigl(2+o(1)\bigr)n\sqrt{\log n},
\end{equation}
and who also proved the general upper bound \(F(n)\ll n^{3/2}\).  In
the same paper they conjectured that
\begin{equation}\label{eq:EP-conj}
  F(n)-f(n,n) \longrightarrow \infty
\end{equation}
as \(n\to\infty\); Erd\H{o}s later offered \(1000\) rupees for a
solution \cite{Erdos92}, and the conjecture forms part of Problem~711
in Bloom's database \cite{Bloom711}, to which we return in
Remark~\ref{rem:711}.  Recently, van Doorn~\cite{vanDoorn} proved the conjecture~\eqref{eq:EP-conj} in
the stronger quantitative form
\begin{equation}\label{eq:vD}
  F(n)-f(n,n)
  > \frac{0.36\, n\log n}{\log\log n}
\end{equation}
for all sufficiently large \(n\).

The purpose of this
note is to show that a smooth-number obstruction going back to
Erd\H{o}s and Pomerance, deployed far from the diagonal, produces
intervals on the larger \(n\log n\) scale.

\begin{theorem}\label{thm:main}
We have
\[
  \liminf_{n\to\infty}
  \frac{F(n)-f(n,n)}{n\log n}
  \ge \frac{1}{\e}.
\]
\end{theorem}

Since \(f(n,n)=o(n\log n)\) by \eqref{eq:EP-bounds}, an equivalent
formulation is \(F(n)\ge(1/\e-o(1))\,n\log n\): for every fixed
\(c<1/\e\) and all sufficiently large \(n\) there is an interval of
length \(c\,n\log n\) containing no system of pairwise distinct
multiples of \(1,2,\ldots,n\).  Because Theorem~\ref{thm:main} operates on the \(n\log n\)
scale, it also immediately implies a strengthening of \eqref{eq:vD} in
the normalization used by van Doorn.

\begin{corollary}\label{cor:arbitrary}
For every fixed \(A>0\), we have, for all sufficiently large \(n\),
\[
  F(n)-f(n,n)
  > \frac{A n\log n}{\log\log n}.
\]
Equivalently,
\[
  \frac{F(n)-f(n,n)}{n\log n/\log\log n}
  \longrightarrow \infty.
\]
\end{corollary}

\begin{remark}[Relation to Erd\H{o}s Problem 711]\label{rem:711}
Problem~711 of Bloom's database \cite{Bloom711} records two questions
drawn from \cite{Erdos92,ErdosPomerance}: whether
\begin{equation}\label{eq:711a}
  F(n)\le n^{1+o(1)},
\end{equation}
a conjecture stated already in \cite{ErdosPomerance}, and whether
\(F(n)-f(n,n)\to\infty\).  (The problem is stated there
for the open interval \((m,m+f(n,m))\); but since all quantities involved
are integers, this changes \(f(n,m)\) by at most \(1\) and is
immaterial at the scales considered here.)  The
second question is the conjecture settled by van Doorn's result~\eqref{eq:vD}, which our Theorem~\ref{thm:main} and
Corollary~\ref{cor:arbitrary} strengthen.  The first remains open, and
Theorem~\ref{thm:main} constrains it from below: combined with the
Erd\H{o}s--Pomerance bound \(F(n)\ll n^{3/2}\), the known bounds now
read
\[
  \left(\frac{1}{\e}-o(1)\right)n\log n
  \;\le\;
  F(n)
  \;\ll\;
  n^{3/2}.
\]
In particular \(F(n)/n\to\infty\), so the conjectured bound cannot be
strengthened to \(F(n)=O(n)\).  An affirmative answer to
\eqref{eq:711a} would therefore be sharp at the level of the exponent.  At
the same time, the method of this paper is consistent with
\eqref{eq:711a}: As explained in the discussion below, in the
regimes considered in this paper the smooth-number obstruction
produces no interval longer than \(\asymp n\log n\), and so places no
barrier against the conjectured
upper bound; we return to the upper bound in
Section~\ref{sec:upper}.

We note finally that matching integers to distinct
multiples in an interval is a recurring theme among Erd\H{o}s's
problems: The closely related Problem~650 \cite{Bloom650} was resolved by van
Doorn, Li and Tang in a recent preprint \cite{vanDoornLiTang}; and
meanwhile, Problem~710 \cite{Bloom710}, which also concerns
\(f(n,n)\), to our knowledge remains open.
\end{remark}

\subsection*{The method, and comparison with diagonal transfer}

The van Doorn proof of \eqref{eq:vD} rests on the elegant diagonal-transfer
inequality
\begin{equation}\label{eq:vanDoorn-ineq}
  kn+f(kn,kn)
  \le
  k^2n+f(n,k^2n)
  \qquad (k,n\in\N)
\end{equation}
\cite[Lemma~2]{vanDoorn}: an admissible system for the indices
\(1,\ldots,n\) in the interval \((k^2n,\,k^2n+f(n,k^2n)]\) extends to
one for the indices \(1,\ldots,kn\) in \((kn,\,k^2n+f(n,k^2n)]\) after
adjoining the multiples \(ki\in(kn,k^2n]\) for \(n<i\le kn\), and the
term \(k^2n\) is the price of bridging the gap between the two
intervals.  The inequality transfers
lower bounds for the diagonal quantity \(f(N,N)\), taken at the larger
point \(N=kn\), into lower bounds for \(F(n)\).  The strength of the
output is therefore governed by the diagonal problem, and diagonal
intervals are genuinely short: \(f(N,N)=O(N\sqrt{\log N})=o(N\log N)\)
by \eqref{eq:EP-bounds}.  Inserting the lower bound of
\eqref{eq:EP-bounds} into \eqref{eq:vanDoorn-ineq} and optimizing over
\(k\) yields \(F(n)-f(n,n)\ge(1/\e-o(1))\,n\log n/\log\log n\), and no
choice of \(k\) extracts more from the transfer as long as
\eqref{eq:EP-bounds} represents the state of knowledge about
\(f(N,N)\); we record this computation in Section~\ref{sec:transfer}.
The \(\log\log n\) in \eqref{eq:vD} is thus inherited from the
diagonal.

The present note reaches the \(n\log n\) scale by leaving the
diagonal. We exhibit explicit starting points
\(m\sim a\,n\log n\) at which \(f(n,m)\) is seen directly to be
large.  The obstruction we use is due to Erd\H{o}s and Pomerance:  Choose
a smoothness threshold \(y=d\log n\) with \(d>0\) fixed, and suppose
that \(1,\ldots,n\) had pairwise distinct multiples
\(a_1,\ldots,a_n\) in an interval \((m,E]\) with \(E\le ny\).  For
every \(y\)-smooth index \(i\in(E/y,n]\), the cofactor \(a_i/i\) is a
positive integer smaller than \(y\), hence it is \(y\)-smooth, and therefore
\(a_i\) is itself \(y\)-smooth.  Consequently \((m,E]\) must contain
at least as many \(y\)-smooth integers as there are \(y\)-smooth
indices in \((E/y,n]\); when it does not, no admissible system exists
and \(f(n,m)>E-m\), as we show in Lemma~\ref{lem:obstruction}.  For \(m=n\) this is
precisely Lemma~1 of \cite{ErdosPomerance}, the engine behind the
diagonal lower bound in \eqref{eq:EP-bounds}---but the freedom in the left
endpoint \(m\) is key to our argument.  In
Section~\ref{sec:diagonal} we confirm that
Lemma~\ref{lem:obstruction} together with the estimates of
Section~\ref{sec:local} recovers that diagonal bound, constant
included.

What makes this comparison favorable at height \(n\log n\) is the
choice \(y\asymp\log n\).  In this regime, the local estimates of
Hildebrand and Tenenbaum (Section~\ref{sec:local}) show that
\(\Psi(x,y)\), the number of \(y\)-smooth integers up to \(x\), behaves
locally like \(x^{\alpha}\) with exponent
\(\alpha\approx\log(1+d)/\log\log n\to0\).  Two consequences drive the
proof.  First, \(\Psi(\cdot,y)\) is nearly invariant under bounded
dilations, so the two counts being compared are of the same
order: each equals \((c_i+o(1))\,\Psi(n,y)/\log\log n\) for explicit
constants \(c_i=c_i(d,a,b)\).  Secondly---and this is the crucial
point---raising the height of the interval from \(n\) to
\(m\sim a\,n\log n\) increases the supply of smooth numbers only by
the bounded factor \((m/n)^{\alpha}\to1+d\) (see
\eqref{eq:m-over-n}); an interval at height \(n\log n\) is essentially
no richer in \(y\)-smooth numbers than a proportionally scaled
interval at height \(n\).  The pigeonhole therefore reduces to an
inequality between constants---condition
\eqref{eq:parameter-condition} on the parameters \((d,a,b)\) in the sequel---which we
optimize in Section~\ref{sec:optimization}.  Already in
\cite{ErdosPomerance} it is remarked that pushing the obstruction
further would require sharp estimates for \(\Psi(x,y)\) with \(y\)
tending to infinity slowly; the saddle-point estimates of Hildebrand
and Tenenbaum \cite{HildebrandTenenbaum}, which appeared six years
later, supply exactly the local control needed here (see also \cite{HildebrandTenenbaumSurvey}).

The \(n\log n\) scale appears to be intrinsic to the method.  One
constraint is unconditional: the hypothesis of
Lemma~\ref{lem:obstruction} forces \(E<ny\) (see its proof), so an
obstruction with threshold \(y\) certifies no interval longer than
\(ny\), and so with \(y\asymp\log n\) no interval longer than
\(\asymp n\log n\) can be obstructed.  Heuristically, raising \(y\) to
a larger power of \(\log n\) does not help: the exponent \(\alpha\)
then stays bounded away from \(0\), and an interval at height
\(n\log n\) already contains more \(y\)-smooth integers than there are
\(y\)-smooth indices below \(n\) by a positive power of \(\log n\);
hence, the smooth targets are too numerous for the counting obstruction
in Lemma~\ref{lem:obstruction} to apply.  Within the three-parameter family optimized in
Section~\ref{sec:optimization}, the constant \(1/\e\) is best possible
as well (Remark~\ref{rem:sup}).

\subsection*{Organization of the paper}

The paper is organized as follows.  Section~\ref{sec:obstruction}
isolates the obstruction.  Section~\ref{sec:local} records the
required estimates for smooth numbers.  Section~\ref{sec:construction}
carries out the construction for a general triple of parameters and
gives a fully explicit example.  Section~\ref{sec:optimization}
optimizes the parameters and proves Theorem~\ref{thm:main} and
Corollary~\ref{cor:arbitrary}.
Section~\ref{sec:remarks} compares the construction with van Doorn's
transfer method and revisits the diagonal case \(m=n\), recovering the
Erd\H{o}s--Pomerance lower bound within the present framework; it then
reformulates the upper-bound problem through Hall's theorem, proves
the conjectured \(n^{1+o(1)}\) upper bound at every polynomial height,
and closes with a question.

\section{A smooth-number obstruction}\label{sec:obstruction}

This section isolates the counting inequality on which all of our
lower bounds rest; it reduces the task of proving \(f(n,m)>E-m\) to a
comparison between two smooth-number counts.  For a positive integer
\(r\), let \(\Pplus(r)\) denote its largest prime
factor, with the convention \(\Pplus(1)=1\).  For real \(x\ge 0\) and
\(y\ge 1\), put
\[
  \Psi(x,y)
  := \#\{r\in\N:r\le x,\ \Pplus(r)\le y\}.
\]
Thus \(\Psi(x,y)\) counts the positive integers up to \(x\) which are
\(y\)-smooth.  For \(m=n\), the following lemma is essentially Lemma~1 of
Erd\H{o}s and Pomerance \cite{ErdosPomerance}, with the same proof;
the only change is that the left endpoint \(m\) of the target interval
is now taken to be a free parameter.

\begin{lemma}[Smooth-number obstruction]\label{lem:obstruction}
Let \(n,m\in\N\), let \(y>1\), and let \(E>m\) be real.  If
\begin{equation}\label{eq:obstruction-condition}
  \Psi(n,y)-\Psi(E/y,y)
  >
  \Psi(E,y)-\Psi(m,y),
\end{equation}
then we have
\[
  f(n,m)>E-m.
\]
\end{lemma}

\begin{proof}
Since \(\Psi(\cdot,y)\) is nondecreasing and the right-hand side of
\eqref{eq:obstruction-condition} is nonnegative, the hypothesis forces
\(\Psi(n,y)>\Psi(E/y,y)\), and in particular \(E/y<n\).  Let
\[
  \mathcal I
  := \{i\in\N:E/y<i\le n,\ \Pplus(i)\le y\},
\]
so that
\[
  |\mathcal I|=\Psi(n,y)-\Psi(E/y,y).
\]
Suppose, for the sake of seeking a contradiction, that \(f(n,m)\le E-m\).  Then there are
pairwise distinct integers \(a_1,\ldots,a_n\in(m,E]\) such that
\(i\mid a_i\) for \(1\le i\le n\).  For \(i\in\mathcal I\), write
\[
  a_i=i q_i, \qquad q_i\in\N.
\]
Since \(a_i\le E\) and \(i>E/y\), we have
\[
  q_i=\frac{a_i}{i}<y.
\]
Therefore \(q_i\) is \(y\)-smooth.  Also \(i\) is \(y\)-smooth by the
definition of \(\mathcal I\), and hence \(a_i=iq_i\) is \(y\)-smooth.
It follows that the distinct integers \(a_i\), with \(i\in\mathcal I\),
all lie in the set
\[
  \{a\in\N:m<a\le E,\ \Pplus(a)\le y\},
\]
which has cardinality \(\Psi(E,y)-\Psi(m,y)\).  Thus
\[
  |\mathcal I|\le \Psi(E,y)-\Psi(m,y),
\]
contrary to \eqref{eq:obstruction-condition}.  Hence no such system of
distinct multiples lies in \((m,E]\), and so we conclude that \(f(n,m)>E-m\).
\end{proof}

The left-hand side of \eqref{eq:obstruction-condition} counts the
smooth indices whose multiples are forced to be smooth, while the
right-hand side counts the smooth landing spots available to them in
the target interval.  A similar interplay between interval matchings
and smooth numbers appears in work on Grimm's conjecture on distinct
prime factors of consecutive integers; see \cite{RST,LaishramMurty}.

The
remainder of the paper consists of choosing
\(m\), \(E\) and \(y\) so that the supply of smooth
landing spots falls short of the demand from the smooth indices; the
estimates needed to compare the two sides are collected in the next
section.

\section{Local estimates for smooth numbers}\label{sec:local}

The two sides of \eqref{eq:obstruction-condition} will be compared by
means of two estimates of Hildebrand and Tenenbaum~\cite{HildebrandTenenbaum}, which describe the local behavior of the
counting function \(\Psi(x,y)\); for an exposition of the underlying
method see also \cite[Chapter~III.5]{Tenenbaum} as well as the survey~\cite{HildebrandTenenbaumSurvey}.  For
\(x\ge y\ge 2\), let \(\alpha(x,y)\)
be the saddle-point parameter, i.e., the unique positive solution
\(\alpha\) of
\begin{equation}\label{eq:alpha-definition}
  \sum_{p\le y}\frac{\log p}{p^{\alpha}-1}=\log x;
\end{equation}
the left-hand side decreases strictly from \(\infty\) to \(0\) as
\(\alpha\) runs over \((0,\infty)\), so \(\alpha(x,y)\) is well
defined.  Uniformly for \(x\ge y\ge 2\) and \(1\le c\le y\), we have
\cite[Theorem~3]{HildebrandTenenbaum}
\begin{equation}\label{eq:HT-ratio}
  \Psi(cx,y)
  =
  \Psi(x,y)c^{\alpha(x,y)}
  \left(
    1+O\left(\frac{1}{u}+\frac{\log y}{y}\right)
  \right),
  \qquad
  u=\frac{\log x}{\log y},
\end{equation}
and, uniformly for \(x\ge y\ge2\)
\cite[Theorem~2(i)]{HildebrandTenenbaum},
\begin{equation}\label{eq:alpha-estimate}
  \alpha(x,y)
  =
  \frac{\log(1+y/\log x)}{\log y}
  \left(
    1+O\left(\frac{\log\log(1+y)}{\log y}\right)
  \right).
\end{equation}

We record the specialization that we use throughout our argument.
Set
\[
  L:=\log n,
  \qquad
  \ell:=\log\log n.
\]

\begin{lemma}\label{lem:alpha-specialization}
Fix \(d>0\), and put \(y=dL\).  Suppose that \(x=x(n)\) satisfies
\[
  c_1 n\le x\le c_2 nL
\]
for some fixed constants \(0<c_1<c_2<\infty\).  Then
\begin{equation}\label{eq:alpha-specialized}
  \alpha(x,y)
  =
  \frac{\log(1+d)}{\ell}
  +o\left(\frac{1}{\ell}\right).
\end{equation}
Furthermore, for these \(x\) and \(y\), uniformly in \(1\le c\le y\),
the relative error in \eqref{eq:HT-ratio} is
\begin{equation}\label{eq:ratio-error}
  O\left(\frac{\ell}{L}\right)
  =o\left(\frac{1}{\ell}\right).
\end{equation}
The estimates are uniform when \(d,c_1,c_2\) range over fixed compact
subsets of \((0,\infty)\), with \(c_1<c_2\).
\end{lemma}

\begin{proof}
In the indicated range for \(x\) we have, for all sufficiently large
\(n\), the inequalities \(x\ge y\ge 2\), together with
\[
  \log x=L+O(\ell),
  \qquad
  \log y=\ell+O(1),
\]
and
\[
  \frac{y}{\log x}
  =\frac{dL}{L+O(\ell)}
  =d+o(1).
\]
Substituting these estimates into \eqref{eq:alpha-estimate} gives
\[
  \alpha(x,y)
  =
  \frac{\log(1+d)+o(1)}{\ell+O(1)}
  \left(1+O\left(\frac{\log\ell}{\ell}\right)\right)
  =
  \left(\frac{\log(1+d)}{\ell}+o\left(\frac{1}{\ell}\right)\right)
  \bigl(1+o(1)\bigr),
\]
which is \eqref{eq:alpha-specialized}.  Also
\[
  u=\frac{\log x}{\log y}\asymp \frac{L}{\ell},
\]
so we have
\[
  \frac{1}{u}+\frac{\log y}{y}
  \ll
  \frac{\ell}{L}+\frac{\ell}{L}
  \ll \frac{\ell}{L}.
\]
Finally \(\ell/L=o(1/\ell)\), because \(\ell^2=o(L)\); this proves
\eqref{eq:ratio-error}.
\end{proof}

The point of \eqref{eq:ratio-error} is that the smooth-number
differences appearing in our analysis have relative size \(\asymp1/\ell\), while
the error in the ratio estimate is of smaller order.  We shall also
use repeatedly, without further comment, that if \(t=t(n)=O(1/\ell)\)
then
\begin{equation}\label{eq:exp-expansion}
  \e^{\pm t}=1\pm t+O(t^2)=1\pm t+o\left(\frac{1}{\ell}\right).
\end{equation}

\section{A three-parameter construction}\label{sec:construction}

With the estimates of Section~\ref{sec:local} in hand, we now apply
the obstruction lemma to a family of intervals described by three
parameters: \(d\) fixes the smoothness threshold \(y=dL\),
while \(a\) and \(b\) locate the interval, which runs from
\(m\approx anL\) to \(E\approx bnL\).  Condition
\eqref{eq:parameter-condition} below expresses that the eligible
smooth indices outnumber the available smooth targets.

\begin{proposition}\label{prop:parameters}
Fix real constants \(d\), \(a\), and \(b\) satisfying
\[
  d>0,
  \qquad
  0<a<b<d,
\]
and
\begin{equation}\label{eq:parameter-condition}
  \log\left(\frac{d}{b}\right)
  >
  (1+d)\log\left(\frac{b}{a}\right).
\end{equation}
Then, with
\[
  m=\left\lfloor an\log n\right\rfloor,
\]
we have
\[
  f(n,m)
  \ge
  \bigl(b-a+o(1)\bigr)n\log n.
\]
\end{proposition}

\begin{proof}
Write
\[
  L=\log n,
  \qquad
  \ell=\log\log n,
  \qquad
  y=dL,
\]
and set
\[
  m=\lfloor anL\rfloor,
  \qquad
  E=\lfloor bnL\rfloor.
\]
As \(0<a<b<d\), for all sufficiently large \(n\) we have \(m<E\),
\(E/y<n\), and \(m/n<y\).  We shall verify the inequality
\eqref{eq:obstruction-condition} and then apply
Lemma~\ref{lem:obstruction}.  In each of the three applications of
\eqref{eq:HT-ratio} in the sequel, the pair \((x,c)\) will be
\((E/y,\,ny/E)\), \((m,\,E/m)\), or \((n,\,m/n)\).  Every base point
satisfies \(\tfrac{b}{2d}\,n\le x\le 2b\,nL\) for all sufficiently
large \(n\), so Lemma~\ref{lem:alpha-specialization} applies to the
three applications simultaneously, with the common fixed constants
\(c_1=\tfrac{b}{2d}\) and \(c_2=2b\); and every dilation factor
satisfies \(1<c<y\), the largest being \(m/n=aL+O(1)<dL=y\).

Put
\[
  \lambda:=\log(1+d).
\]

\smallskip
\noindent\textit{Eligible indices.}
Apply \eqref{eq:HT-ratio} with
\[
  x=\frac{E}{y},
  \qquad
  c=\frac{ny}{E}.
\]
Here \(x\asymp n\), and
\[
  c=\frac{ny}{E}=\frac{d}{b}+o(1)\in(1,y).
\]
By Lemma~\ref{lem:alpha-specialization},
\[
  \alpha(E/y,y)=\frac{\lambda}{\ell}+o\left(\frac{1}{\ell}\right),
\]
and the relative error in \eqref{eq:HT-ratio} is \(o(1/\ell)\).  Thus
\[
  \Psi(n,y)
  =
  \Psi(E/y,y)c^{\alpha(E/y,y)}
  \left(1+o\left(\frac{1}{\ell}\right)\right),
\]
or equivalently
\[
  \frac{\Psi(E/y,y)}{\Psi(n,y)}
  =
  c^{-\alpha(E/y,y)}
  \left(1+o\left(\frac{1}{\ell}\right)\right).
\]
Since \(\log c=\log(d/b)+o(1)\), we get
\(\alpha(E/y,y)\log c=\lambda\log(d/b)/\ell+o(1/\ell)\), whence by
\eqref{eq:exp-expansion} we have
\[
  c^{-\alpha(E/y,y)}
  =
  1-\frac{\lambda\log(d/b)}{\ell}
  +o\left(\frac{1}{\ell}\right).
\]
Therefore
\begin{equation}\label{eq:left-difference}
  \Psi(n,y)-\Psi(E/y,y)
  =
  \left(
    \frac{\lambda\log(d/b)}{\ell}
    +o\left(\frac{1}{\ell}\right)
  \right)\Psi(n,y).
\end{equation}

\smallskip
\noindent\textit{Available smooth integers.}
Now, we apply \eqref{eq:HT-ratio} with
\[
  x=m,
  \qquad
  c=\frac{E}{m}.
\]
Here \(x\asymp nL\), and
\[
  c=\frac{b}{a}+o(1)\in(1,y).
\]
Lemma~\ref{lem:alpha-specialization} gives
\[
  \alpha(m,y)=\frac{\lambda}{\ell}+o\left(\frac{1}{\ell}\right),
\]
and hence, by \eqref{eq:exp-expansion} again, we have
\begin{align}
  \Psi(E,y)-\Psi(m,y)
  &=
  \Psi(m,y)
  \left(
    c^{\alpha(m,y)}
    -1
    +o\left(\frac{1}{\ell}\right)
  \right) \notag\\
  &=
  \left(
    \frac{\lambda\log(b/a)}{\ell}
    +o\left(\frac{1}{\ell}\right)
  \right)\Psi(m,y).
  \label{eq:right-first}
\end{align}

\smallskip
\noindent\textit{Packing comparison.} It remains to compare \(\Psi(m,y)\) with \(\Psi(n,y)\).  For this, we apply
\eqref{eq:HT-ratio} with
\[
  x=n,
  \qquad
  c=\frac{m}{n};
\]
now, \(c=aL+o(1)\), and the inequality \(a<d\) gives \(1<c<y=dL\)
for all sufficiently large \(n\).  Hence \eqref{eq:HT-ratio} gives
\[
  \frac{\Psi(m,y)}{\Psi(n,y)}
  =
  \left(\frac{m}{n}\right)^{\alpha(n,y)}
  \left(1+o\left(\frac{1}{\ell}\right)\right).
\]
Furthermore, we have
\[
  \log\left(\frac{m}{n}\right)=\ell+\log a+o(1)
\]
and
\[
  \alpha(n,y)=\frac{\lambda}{\ell}+o\left(\frac{1}{\ell}\right),
\]
so
\[
  \alpha(n,y)\log\left(\frac{m}{n}\right)=\lambda+o(1).
\]
Consequently,
\begin{equation}\label{eq:m-over-n}
  \frac{\Psi(m,y)}{\Psi(n,y)}
  =
  \e^{\lambda}+o(1)
  =
  1+d+o(1).
\end{equation}
Combining \eqref{eq:right-first} and \eqref{eq:m-over-n}, we obtain
\begin{equation}\label{eq:right-difference}
  \Psi(E,y)-\Psi(m,y)
  =
  \left(
    \frac{(1+d)\lambda\log(b/a)}{\ell}
    +o\left(\frac{1}{\ell}\right)
  \right)\Psi(n,y).
\end{equation}

\smallskip
\noindent\textit{Conclusion.}
Subtracting \eqref{eq:right-difference} from \eqref{eq:left-difference}
gives
\begin{align*}
  &\Psi(n,y)-\Psi(E/y,y)-\bigl(\Psi(E,y)-\Psi(m,y)\bigr) \\
  &\qquad =
  \frac{\lambda\Psi(n,y)}{\ell}
  \left(
    \log\left(\frac{d}{b}\right)
    -(1+d)\log\left(\frac{b}{a}\right)
    +o(1)
  \right);
\end{align*}
by \eqref{eq:parameter-condition}, the right-hand side is positive for
all sufficiently large \(n\).  Thus the obstruction condition~\eqref{eq:obstruction-condition} holds, and Lemma~\ref{lem:obstruction}
gives
\[
  f(n,m)>E-m.
\]
Since
\[
  E-m=(b-a)nL+O(1),
\]
the proposition follows.
\end{proof}

Proposition~\ref{prop:parameters} leaves the triple \((d,a,b)\) free.
Before optimizing it in Section~\ref{sec:optimization}, we record a
simple admissible choice, which already produces intervals on the
\(n\log n\) scale.

\begin{corollary}\label{cor:explicit}
If we take
\[
  m=\left\lfloor \frac{2}{5} n\log n\right\rfloor,
\]
then
\[
  f(n,m)
  \ge
  \left(\frac{1}{10}+o(1)\right)n\log n.
\]
Consequently,
\[
  F(n)-f(n,n)
  \ge
  \left(\frac{1}{10}-o(1)\right)n\log n.
\]
\end{corollary}

\begin{proof}
In Proposition~\ref{prop:parameters}, take
\[
  d=1,
  \qquad
  a=\frac{2}{5},
  \qquad
  b=\frac{1}{2}.
\]
The parameter condition \eqref{eq:parameter-condition} becomes
\[
  \log 2>2\log\left(\frac{5}{4}\right),
\]
which holds because
\[
  2>\left(\frac{5}{4}\right)^2=\frac{25}{16}.
\]
Thus Proposition~\ref{prop:parameters} gives the asserted lower bound
for \(f(n,m)\).  By the upper bound in \eqref{eq:EP-bounds} we have
\(f(n,n)=O(n\sqrt{\log n})=o(n\log n)\), and the lower bound for
\(F(n)-f(n,n)\) follows.
\end{proof}

Corollary~\ref{cor:explicit} already implies
Corollary~\ref{cor:arbitrary}; the optimization in the next section is
needed only for the constant \(1/\e\) in Theorem~\ref{thm:main}.

\section{Optimizing the parameters}\label{sec:optimization}

It remains to choose the parameters in
Proposition~\ref{prop:parameters} optimally; we begin by determining
the admissible range of \(a\).  For fixed \(d\) and \(b\) with
\(0<b<d\), condition
\eqref{eq:parameter-condition} is equivalent to
\[
  \log\left(\frac{b}{a}\right)
  <
  \frac{1}{1+d}\log\left(\frac{d}{b}\right),
\]
that is, to
\begin{equation}\label{eq:a-boundary}
  a>a_0(d,b):=
  b\left(\frac{b}{d}\right)^{1/(1+d)}.
\end{equation}
Note that \(0<a_0(d,b)<b\).  Hence, for fixed \(d\) and \(b\), the
supremum of the interval coefficients \(b-a\) obtainable from
Proposition~\ref{prop:parameters} is
\[
  b\left(1-\left(\frac{b}{d}\right)^{1/(1+d)}\right).
\]
Write
\[
  r=\frac{b}{d},
  \qquad
  0<r<1;
\]
then the supremal coefficient is
\[
  d r\left(1-r^{1/(1+d)}\right).
\]

\begin{lemma}\label{lem:optimization}
For fixed \(d>0\),
\[
  \max_{0<r<1}\,
  d r\left(1-r^{1/(1+d)}\right)
  = C_d,
\]
where
\begin{equation}\label{eq:Cd}
  C_d
  :=
  \frac{d}{d+2}
  \left(\frac{d+1}{d+2}\right)^{d+1},
\end{equation}
the maximum being attained at
\(r=\bigl((d+1)/(d+2)\bigr)^{d+1}\).  Moreover,
\[
  \lim_{d\to\infty} C_d=\frac{1}{\e}.
\]
\end{lemma}

\begin{proof}
Put
\[
  s=\frac{1}{1+d}.
\]
It suffices to maximize
\[
  g(r)=r-r^{1+s}
  \qquad (0<r<1).
\]
Since
\[
  g'(r)=1-(1+s)r^s,
\]
the unique critical point of \(g\) in \((0,1)\) is determined by
\[
  r^s=\frac{1}{1+s}=\frac{d+1}{d+2},
\]
that is,
\begin{equation}
\label{eq:g-critical}
  r=
  \left(\frac{d+1}{d+2}\right)^{d+1}.
\end{equation}
As \(g\) is continuous on \([0,1]\), positive on \((0,1)\), and
vanishes at both endpoints, the critical point~\eqref{eq:g-critical} is the global maximum.
At this value,
\[
  1-r^{1/(1+d)}
  =
  1-\frac{d+1}{d+2}
  =
  \frac{1}{d+2},
\]
and \eqref{eq:Cd} follows.  Finally,
\[
  \frac{d}{d+2}\longrightarrow 1
\]
and
\[
  \left(\frac{d+1}{d+2}\right)^{d+1}
  =
  \left(1-\frac{1}{d+2}\right)^{d+1}
  \longrightarrow \e^{-1}.
\]
\end{proof}

\begin{remark}\label{rem:sup}
The constant \(1/\e\) is the supremum of what
Proposition~\ref{prop:parameters} can deliver, and it is not attained
for any finite \(d\).  Indeed, from \(1-r^{1/(1+d)}<\log(1/r)/(1+d)\)
for \(0<r<1\) we get
\[
  C_d
  <
  \frac{d}{d+1}\,\max_{0<r<1}\, r\log\frac{1}{r}
  =
  \frac{d}{(d+1)\e}
  <\frac{1}{\e}
  \qquad(d>0),
\]
while \(C_d\to1/\e\) by Lemma~\ref{lem:optimization}.  In the limit as
\(d\to\infty\), the optimization degenerates to maximizing
\(r\log(1/r)\), whose maximum \(1/\e\) is attained at \(r=1/\e\);
correspondingly, the optimal ratio
\(r=\bigl((d+1)/(d+2)\bigr)^{d+1}\) of Lemma~\ref{lem:optimization}
tends to \(1/\e\), yielding the constant in
Theorem~\ref{thm:main}.

We stress the resulting order of quantifiers
in Theorem~\ref{thm:main}: a coefficient \(C\) near \(1/\e\) requires
\(d\) large, and at the optimizing choice we have \(b=dr\sim d/\e\)
and \(a>a_0\sim d/\e\), so the certified intervals begin at heights
\(m\sim a\,nL\) with \(a\to\infty\) as \(C\to1/\e\).  No endpoint
construction with \(d=d(n)\to\infty\) is asserted---indeed
Lemma~\ref{lem:alpha-specialization} is a fixed-parameter estimate.
\end{remark}

We can now prove the main theorem and its scale corollary.

\begin{proof}[Proof of Theorem~\ref{thm:main}]
Let \(0<C<1/\e\).  By Lemma~\ref{lem:optimization}, choose \(d>0\)
so large that
\[
  C<C_d.
\]
Let
\[
  r=
  \left(\frac{d+1}{d+2}\right)^{d+1},
  \qquad
  b=dr,
\]
so that \(0<b<d\), and set \(a_0=a_0(d,b)\) as in
\eqref{eq:a-boundary}.  Then
\[
  b-a_0=b\left(1-\left(\frac{b}{d}\right)^{1/(1+d)}\right)
  =dr\left(1-r^{1/(1+d)}\right)=C_d,
\]
and, by \eqref{eq:a-boundary}, the triple \((d,a,b)\) satisfies
\eqref{eq:parameter-condition} precisely when \(a\in(a_0,b)\).  We may
therefore choose \(a\in(a_0,b)\) so close to \(a_0\) that
\[
  b-a>C.
\]
Proposition~\ref{prop:parameters} then implies that
\[
  F(n)
  \ge
  \bigl(b-a+o(1)\bigr)n\log n.
\]
By the upper bound in \eqref{eq:EP-bounds},
\[
  f(n,n)=O(n\sqrt{\log n})=o(n\log n).
\]
Therefore
\[
  \liminf_{n\to\infty}
  \frac{F(n)-f(n,n)}{n\log n}
  \ge b-a>C.
\]
Since \(C<1/\e\) was arbitrary, the theorem follows.
\end{proof}

\begin{proof}[Proof of Corollary~\ref{cor:arbitrary}]
By Theorem~\ref{thm:main}, for all sufficiently large \(n\),
\[
  F(n)-f(n,n)
  \ge
  \frac{1}{1754\e} n\log n.
\]
For fixed \(A>0\), we have \(A/\log\log n<1/(1754\e)\) for all
sufficiently large \(n\).\footnote{The choice \(1/(1754\e)\) here is
purely aesthetic; we could use any fixed positive constant less than
\(1/\e\).}  Hence
\[
  F(n)-f(n,n)
  > \frac{A n\log n}{\log\log n}
\]
for all sufficiently large \(n\); the displayed limiting form is
immediate.
\end{proof}

\section{Remarks}\label{sec:remarks}

\subsection{Relation to van Doorn's diagonal-transfer framework}\label{sec:transfer}

We now record a comparison with the transfer inequality
\eqref{eq:vanDoorn-ineq}.

Combining \eqref{eq:vanDoorn-ineq} with
\(F(n)\ge f(n,k^2n)\) gives, for every \(k\in\N\),
\[
  F(n)-f(n,n)
  \ge
  kn+f(kn,kn)-k^2n-f(n,n).
\]
Write \(L=\log n\) and \(\ell=\log\log n\).  Insert the lower bound of
\eqref{eq:EP-bounds} at \(N=kn\), and first take
\[
  k=\left\lfloor\kappa\sqrt{\frac{L}{\ell}}\right\rfloor
\]
with \(\kappa>0\) fixed.  Then
\[
  kn=o\left(\frac{nL}{\ell}\right)
\]
and, by the upper bound in \eqref{eq:EP-bounds},
\[
  f(n,n)=O(n\sqrt L)
  =o\left(\frac{nL}{\ell}\right).
\]
Moreover,
\[
  \log(kn)=(1+o(1))L,
  \qquad
  \log\log(kn)=(1+o(1))\ell.
\]
It follows that
\[
  f(kn,kn)
  \ge
  \left(\frac{2\kappa}{\sqrt{\e}}+o(1)\right)
  \frac{nL}{\ell},
  \qquad
  k^2n
  =
  \left(\kappa^2+o(1)\right)
  \frac{nL}{\ell}.
\]
Consequently,
\[
  F(n)-f(n,n)
  \ge
  \left(
    \frac{2\kappa}{\sqrt{\e}}-\kappa^2-o(1)
  \right)
  \frac{nL}{\ell}.
\]
The coefficient
\[
  \frac{2\kappa}{\sqrt{\e}}-\kappa^2
\]
is maximized at \(\kappa=1/\sqrt{\e}\), where its value is \(1/\e\).

With only the diagonal estimates \eqref{eq:EP-bounds} inserted into \eqref{eq:vanDoorn-ineq}, the optimization can neither improve the
\(nL/\ell\) scale, nor the constant.  Indeed, suppose first that \(k\le3\sqrt L\)
and put
\[
  \kappa_n:=k\sqrt{\frac{\ell}{L}}.
\]
If \(\kappa_n\) remains bounded along a subsequence, the preceding
calculation applies along that subsequence and gives a coefficient at
most \(1/\e+o(1)\).  If instead \(\kappa_n\to\infty\), then the positive
term supplied by the diagonal lower estimate is
\[
  O(\kappa_n)\frac{nL}{\ell},
\]
whereas
\[
  k^2n=\kappa_n^2\frac{nL}{\ell};
\]
the resulting lower bound is therefore eventually negative.

For \(k\ge3\sqrt L\), the upper bound in \eqref{eq:EP-bounds}, now
applied at \(N=kn\), shows that the transfer expression itself is
nonpositive for all sufficiently large \(n\).  Indeed,
\[
  f(kn,kn)
  \le
  \bigl(2+o(1)\bigr)kn\sqrt{\log(kn)},
\]
while \(k\ge3\sqrt L\) implies, uniformly in this range,
\[
  1+\bigl(2+o(1)\bigr)\sqrt{\log(kn)}<k.
\]
Indeed, \((L+\log k)/k^{2}\) is decreasing in \(k\), so the ratio
\(\sqrt{\log(kn)}/k\) is largest at the left endpoint
\(k=\lceil3\sqrt L\,\rceil\), where it equals \(\tfrac{1}{3}+o(1)\); thus
\(\sqrt{\log(kn)}\le\bigl(\tfrac{1}{3}+o(1)\bigr)k\) throughout the range.
Hence
\[
  kn+f(kn,kn)-k^2n<0.
\]

Thus the transfer inequality, when supplied with the currently known
diagonal estimates, yields at most
\[
  F(n)-f(n,n)
  \ge
  \left(\frac{1}{\e}-o(1)\right)\frac{nL}{\ell},
\]
which is exactly the same \(nL/\ell\) scale obtained by van Doorn, with the
numerical constant sharpened from \(0.36\) to \(1/\e\) by optimizing
the transfer parameter.  The bottleneck is the diagonal lower bound:
the quadratic transfer squares the \(\sqrt{\log\log n}\) gap in
\eqref{eq:EP-bounds}, and this route could reach the \(n\log n\) scale
only if the diagonal lower bound were improved to
\(\asymp n\sqrt{\log n}\).  The off-diagonal obstruction leveraged in this paper
bypasses the diagonal---and thus sidesteps that barrier.

At present we do not know whether the numerical agreement between
the transfer optimum and the constant in Theorem~\ref{thm:main} is a
coincidence. That said, given the recurring role of smooth-number phenomena in
both arguments, it would not be surprising if the concordance of constants reflected a deeper
connection.

\subsection{The diagonal bound revisited}\label{sec:diagonal}

Lemma~\ref{lem:obstruction} in its original diagonal case \(m=n\),
combined with the estimates of Section~\ref{sec:local}, also recovers
the lower bound of \eqref{eq:EP-bounds}, with the same constant.  We
include the short computation because it locates the source of the
constant \(2/\sqrt{\e}\) and clarifies how the diagonal and
off-diagonal regimes of the obstruction differ.  The optimal
threshold below is \(y\asymp\log n/\log\log n\), outside the regime
\(y=d\log n\) of Lemma~\ref{lem:alpha-specialization}, so we argue
directly from \eqref{eq:HT-ratio} and \eqref{eq:alpha-estimate}.

\begin{proposition}[Erd\H{o}s--Pomerance
{\cite[Theorem~2]{ErdosPomerance}}]\label{prop:diagonal}
As \(n\to\infty\),
\[
  f(n,n)\;\ge\;\left(\frac{2}{\sqrt{\e}}+o(1)\right)
  n\sqrt{\frac{\log n}{\log\log n}}.
\]
\end{proposition}

\begin{proof}
Write \(L=\log n\) and \(\ell=\log\log n\).  Fix
\(\varepsilon\in(0,1)\) and set
\[
  y=\frac{4L}{\ell},
  \qquad
  K=\left\lfloor(1-\varepsilon)\,\frac{2}{\sqrt{\e}}
    \sqrt{\frac{L}{\ell}}\right\rfloor,
  \qquad
  E=Kn,
\]
so that \(1<K<\sqrt{y}\) and \(y<E/y<n\) for all sufficiently large
\(n\).  Put \(s:=\log y=\ell-\log\ell+\log4+o(1)\), and let
\(\alpha:=\alpha(E/y,\,y)\) and \(t:=\alpha s\), so that
\(y^{-\alpha}=\e^{-t}\).

Since \(\log(E/y)=L+O(\ell)\), the Hildebrand--Tenenbaum estimate \eqref{eq:alpha-estimate}
gives
\[
  t=\log\left(1+\frac{y}{\log(E/y)}\right)
    \left(1+O\left(\frac{\log\ell}{\ell}\right)\right)
   =\frac{4}{\ell}\left(1+O\left(\frac{\log\ell}{\ell}\right)\right),
\]
so that \(\alpha=t/s\asymp\ell^{-2}\). Moreover, in
\eqref{eq:HT-ratio} with base \(x=E/y\), the relative error is
\[
  \eta:=O\left(\frac{1}{u}+\frac{\log y}{y}\right)
      =O\left(\frac{\ell^{2}}{L}\right)=o(\alpha),
\]
since \(u=\log(E/y)/\log y\asymp L/\ell\) and \(\ell^{4}=o(L)\).

We apply \eqref{eq:HT-ratio} twice with the single base \(x=E/y\)---once
with \(c=y\) and once with \(c=y/K\in(1,y)\)---using \((E/y)\,y=E\) and
\((E/y)(y/K)=n\):
\[
  \Psi(E,y)=\Psi(E/y,y)\,y^{\alpha}\bigl(1+O(\eta)\bigr),
  \qquad
  \Psi(n,y)=\Psi(E/y,y)\,(y/K)^{\alpha}\bigl(1+O(\eta)\bigr).
\]
The hypothesis \eqref{eq:obstruction-condition} of
Lemma~\ref{lem:obstruction}, with \(m=n\), is
\[
  2\Psi(n,y)>\Psi(E,y)+\Psi(E/y,y).
\]
Substituting the two ratio estimates above and dividing through by
\(\Psi(E/y,y)(y/K)^\alpha\), we find an absolute constant \(C_0>0\)
such that the obstruction condition follows from
\[
  2\bigl(1-C_0\eta\bigr)
  >
  K^\alpha\bigl(1+\e^{-t}\bigr)+C_0\eta\,K^\alpha.
\]
Since \(K<\sqrt y\), we have
\[
  \alpha\log K<\frac{t}{2}=o(1),
\]
and hence \(K^\alpha=1+o(1)\).  It therefore suffices to prove
\begin{equation}\label{eq:diag-condition}
  K^\alpha\bigl(1+\e^{-t}\bigr)<2-C_1\eta
\end{equation}
for a suitable absolute constant \(C_1>0\); taking logarithms, this
is in turn implied by
\[
  \alpha\log K<\varphi(t)-C_2\eta,
  \qquad
  \varphi(t):=\log\frac{2}{1+\e^{-t}}.
\]
Since \(\varphi(t)=\tfrac{t}{2}-\tfrac{t^{2}}{8}+O(t^{3})\), while
\(st=4+O(\log\ell/\ell)\), \(st^{2}=O(1/\ell)\) and
\(\eta/\alpha=o(1)\), the resulting admissible range for \(\log K\)
has right endpoint
\[
  \frac{\varphi(t)}{\alpha}-\frac{C_2\eta}{\alpha}
  =\frac{s}{2}-\frac{st}{8}+O\bigl(st^{2}\bigr)-o(1)
  =\frac{\ell-\log\ell+\log4}{2}-\frac{1}{2}+o(1)
  =\log\left(\frac{2}{\sqrt{\e}}\sqrt{\frac{L}{\ell}}\right)+o(1).
\]
Our choice satisfies
\(\log K\le\log(1-\varepsilon)
+\log\bigl(\tfrac{2}{\sqrt{\e}}\sqrt{L/\ell}\bigr)\),
and \(\log(1-\varepsilon)\) is a negative constant while
\(\eta=o(\alpha)\), so
\eqref{eq:diag-condition} holds for all sufficiently large \(n\).
Lemma~\ref{lem:obstruction} then yields
\[
  f(n,n)>E-n=(K-1)n
  \ge\left((1-\varepsilon)\frac{2}{\sqrt{\e}}+o(1)\right)
  n\sqrt{\frac{L}{\ell}}.
\]
As \(\varepsilon\in(0,1)\) was chosen arbitrarily, the proposition follows.
\end{proof}

Two features of the computation deserve comment.  First, the optimal
parameters differ structurally from those of
Section~\ref{sec:construction}.  On the diagonal, the threshold is
\[
  y\asymp\frac{4\log n}{\log\log n},
\]
the saddle-point exponent satisfies
\[
  \alpha\asymp(\log\log n)^{-2},
\]
and the endpoint ratio
\[
  \frac{E}{n}=K\asymp\sqrt y
\]
is unbounded.  Off the diagonal, by contrast, the threshold is
\(y=d\log n\), the exponent is
\(\alpha\asymp1/\log\log n\), and the ratios entering the first two
local comparisons remain bounded.

Second, within the main-term version of the diagonal smooth-number
comparison, there is a natural ceiling.  Disregarding the error
factors in \eqref{eq:HT-ratio} and \eqref{eq:alpha-estimate}, the
smooth-number obstruction condition becomes
\[
  \alpha\log K<\varphi(t);
\]
since
\[
  \varphi(t)<\frac{t}{2}
  \qquad(t>0),
\]
this main-term condition forces \(K<\sqrt y\).  The resulting
second-order optimization, carried out in the proof of
Proposition~\ref{prop:diagonal}, selects
\[
  y\sim\frac{4\log n}{\log\log n}
\]
and yields the constant \(2/\sqrt{\e}\).  Thus, within the diagonal
implementation of the smooth-number obstruction, the
\(\sqrt{\log\log n}\) gap in \eqref{eq:EP-bounds} appears intrinsic.
The improvement of the diagonal bound contemplated in
Section~\ref{sec:transfer} would therefore appear to require going beyond the present
obstruction method.

\subsection{Toward the upper bound}\label{sec:upper}

Lemma~\ref{lem:obstruction} can be interpreted as one half of an exact
matching duality.  Form the bipartite graph joining each
\(i\in\{1,\ldots,n\}\) to the integers \(j\in(m,m+h]\) satisfying
\(i\mid j\), and for \(S\subseteq\{1,\ldots,n\}\) write
\[
  \Gamma(S)
  :=
  \bigl\{\,j\in(m,m+h]:
       \text{\(i\mid j\) for some \(i\in S\)}\,\bigr\}
\]
for the neighborhood of \(S\).  A system of pairwise
distinct multiples is precisely a matching saturating
\(\{1,\ldots,n\}\).  By Hall's marriage theorem \cite{Hall}, we have \(f(n,m)\le h\) if and only if
\[
  |\Gamma(S)|\ge|S|
  \qquad\text{for every \(S\subseteq\{1,\ldots,n\}\)};
\]
this is the route by which the upper bounds of
\cite{ErdosPomerance} are proven.  Lemma~\ref{lem:obstruction} says
that, when \eqref{eq:obstruction-condition} holds, the \(y\)-smooth
integers in \((E/y,n]\) violate Hall's condition in the window
\((m,E]\)---all of their neighbors in that window are \(y\)-smooth.

Upper bounds for \(F(n)\) can therefore be viewed as structure
theorems for Hall violators.  Such a violator may always be considered to
be closed under taking multiples within \(\{1,\ldots,n\}\).  Indeed,
for \(S\subseteq\{1,\ldots,n\}\), define
\[
  S^\uparrow
  :=
  \{j\le n:\text{\(i\mid j\) for some \(i\in S\)}\}.
\]
Every neighbor of an element of \(S^\uparrow\) is already a neighbor
of an element of \(S\), and therefore
\[
  \Gamma(S^\uparrow)=\Gamma(S).
\]
Since \(|S^\uparrow|\ge|S|\), replacing \(S\) by \(S^\uparrow\)
preserves, and may strengthen, any violation of Hall's condition.

The proof of \cite[Theorem~6]{ErdosPomerance} contains a strong
partial converse.  Write \(d_n(j)\) for the number of divisors of
\(j\) lying in \([1,n]\), and let \(T\) be a positive integer.  If no
integer of the window \((m,\,m+nT]\) has \(d_n(j)>T\), then in the
bipartite graph described above every index \(i\le n\) has at least
\(\lfloor nT/i\rfloor\ge T\) neighbors, while every \(j\) has at most
\(T\) neighbors.  Hence the standard valence form of Hall's theorem
\cite[Theorem~5.2]{BondyMurty}---that a bipartite graph has a matching
saturating the left side whenever every left degree is at least \(T\)
and every right degree is at most \(T\)---gives Hall's condition at once,
and \(f(n,m)\le nT\).  Erd\H{o}s and Pomerance applied this with \(T\) of
size
\[
  D(n):=\exp\bigl((\beta+o(1))\log n/\log\log n\bigr),
  \qquad
  \beta=\log\tfrac{5}{2}+\tfrac{3}{2}\log\tfrac{5}{3}=1.6825\ldots;
\]
their count of the integers with abnormally many divisors below \(n\)
then shows that \(f(n,m)\le n\,D(n)=n^{1+o(1)}\) for all but an
\(n^{-1/2+o(1)}\) proportion of the residues \(m\) modulo
\(\lcm(1,\ldots,n)\).  The same valence count, with a threshold adapted
to the height of the window, establishes \eqref{eq:711a} at every
polynomial height.
\begin{proposition}[after Erd\H{o}s--Pomerance]\label{prop:polyheight}
For every fixed \(C>0\),
\[
  \max_{m\le n^{C}}f(n,m)\;\le\;n^{1+o(1)}.
\]
\end{proposition}

\begin{proof}
We may assume \(C\ge1\), since the left-hand side of the claimed bound is nondecreasing in
\(C\).  Write
\[
  G(\tau)
  :=
  \log(1+\tau)+\tau\log\left(1+\frac{1}{\tau}\right),
\]
and set
\[
  T
  :=
  \left\lceil
    \exp\left(
      (G(C)+1)\frac{\log n}{\log\log n}
    \right)
  \right\rceil.
\]
Thus \(T=n^{o(1)}\).  Every integer in a window
\((m,m+nT]\) with \(m\le n^C\) satisfies
\[
  j\le n^C+nT\le n^{C+o(1)}.
\]

Let \(q_1<\cdots<q_\omega\) be the distinct prime factors of \(j\) that
do not exceed \(n\), and let \(p_i\) denote the \(i\)-th prime.  If
\(\omega=0\), then the only divisor of \(j\) lying in \([1,n]\) is
\(1\), so \(d_n(j)=1<T\) for all sufficiently large \(n\); we may
therefore assume \(\omega\ge1\).  Every divisor \(r\le n\) of \(j\)
involves only the primes \(q_1,\ldots,q_\omega\)---a prime factor of
\(j\) exceeding \(n\) cannot divide \(r\), since it alone would
already make \(r>n\)---and hence can be written uniquely in the form
\[
  r=q_1^{e_1}\cdots q_\omega^{e_\omega}.
\]
Replacing each \(q_i\) by \(p_i\) defines an injection
\[
  q_1^{e_1}\cdots q_\omega^{e_\omega}
  \longmapsto
  p_1^{e_1}\cdots p_\omega^{e_\omega}
\]
from the divisors of \(j\) in \([1,n]\) to the \(p_\omega\)-smooth
integers in \([1,n]\).  Indeed, \(p_i\le q_i\) for every \(i\), and
therefore
\[
  p_1^{e_1}\cdots p_\omega^{e_\omega}
  \le
  q_1^{e_1}\cdots q_\omega^{e_\omega}
  \le n.
\]
Consequently,
\begin{equation}\label{eq:dn-smooth-bound}
  d_n(j)\le\Psi(n,p_\omega).
\end{equation}

We also have
\[
  j\ge q_1q_2\cdots q_\omega
  \ge p_1p_2\cdots p_\omega.
\]
If \(p_\omega\) remains bounded, then \(p_\omega\le(C+o(1))\log n\)
holds \textit{a priori}.  Otherwise, by the Prime Number Theorem in the form
\(\vartheta(x)\sim x\),
\[
  \log(p_1p_2\cdots p_\omega)
  =\vartheta(p_\omega)
  =(1+o(1))p_\omega.
\]
Since \(j\le n^{C+o(1)}\), it follows in either case that
\[
  p_\omega\le(C+o(1))\log n.
\]

Choose a fixed \(\delta>0\) sufficiently small that
\[
  G(C+\delta)<G(C)+\frac{1}{2}.
\]
For all sufficiently large \(n\), we then have
\[
  p_\omega\le(C+\delta)\log n.
\]
By monotonicity and the classical de~Bruijn estimate \cite{deBruijn}
\[
  \log\Psi(x,\tau\log x)
  =
  \bigl(G(\tau)+o(1)\bigr)
  \frac{\log x}{\log\log x}
  \qquad(\text{\(\tau>0\) fixed})
\]
(see also \cite[Chapter~III.5]{Tenenbaum}), we obtain from \eqref{eq:dn-smooth-bound}
\begin{align*}
  d_n(j)
  &\le
  \Psi\bigl(n,(C+\delta)\log n\bigr)\\
  &=
  \exp\left(
    \bigl(G(C+\delta)+o(1)\bigr)
    \frac{\log n}{\log\log n}
  \right)\\
  &<
  \exp\left(
    (G(C)+1)\frac{\log n}{\log\log n}
  \right)\\
  &\le T
\end{align*}
for all sufficiently large \(n\), uniformly for
\(m\le n^C\) and \(j\in(m,m+nT]\).

It remains to apply the valence form of Hall's theorem cited above.
Every index \(i\le n\) has at least
\[
  \left\lfloor\frac{nT}{i}\right\rfloor\ge T
\]
neighbors in \((m,m+nT]\), whereas every integer \(j\) in that window
has fewer than \(T\) neighbors among the indices \(1,\ldots,n\).
Hence there is a matching saturating \(\{1,\ldots,n\}\).  Therefore
\[
  f(n,m)\le nT=n^{1+o(1)}
\]
uniformly for \(m\le n^C\).  The \(o(1)\) may depend on the fixed
constant \(C\).
\end{proof}

All ingredients of the proof of Proposition~\ref{prop:polyheight} are already in \cite{ErdosPomerance},
where they drive their Theorem~6.  The one
additional observation here is the height bound
\[
  p_\omega\le(C+o(1))\log n,
\]
which eliminates the exceptional set of that argument throughout
every fixed polynomial range.

What remains of \eqref{eq:711a} is thus genuinely a question about
windows at super-polynomial heights, positioned near the sparse
integers having abnormally many divisors in \([1,n]\).

The long intervals obtained in this paper nevertheless may offer a hint about the extremal
violators.  Our long intervals occur already at height \(m\asymp n\log n\), far
below the unresolved range, and, as explained in the Introduction, in the regimes we have considered, the smooth-number obstruction produces no interval longer
than \(\asymp n\log n\).

\begin{question}
Is \(F(n)\ll n\log n\)?  In structural terms: is the smooth family of
Lemma~\ref{lem:obstruction}, up to constants, the extremal Hall
violator?
\end{question}

The structural half of the question is posed broadly on purpose: a
Hall violator can a priori be any set closed under taking multiples
within \(\{1,\ldots,n\}\), not only a smooth family.  What can be said
in general is that a deficit of neighbors in a window can occur only
when the sets \(\{j:i\mid j\}\) and \(i\in S\) overlap heavily there---and
such overlap imposes a strongly multiplicative constraint on \(S\).
Smoothness is a mechanism for producing such overlap, but other mechanisms may be possible.

An upper bound \(F(n)\le(1/\e+o(1))\,n\log n\) would combine with
Theorem~\ref{thm:main} to give
\(F(n)=(1/\e+o(1))\,n\log n\); we stress, however, that \(1/\e\) is
known to be extremal only within the three-parameter family of
Section~\ref{sec:construction} (see Remark~\ref{rem:sup}), so on the
present evidence the constant is speculative.  In the opposite
direction, a comment of Tao in the discussion thread
for Problem~711 \cite{Tao711,Bloom711} suggests that \eqref{eq:711a} may,
in some sense at least, imply the Kakeya conjecture in all dimensions via
a discretization of tubes into arithmetic progressions; the upper-bound problem should in
any case be expected to be difficult.  A more tractable intermediate goal may be
to improve the exponent \(3/2\), or to establish the \(1/\e\) ceiling
over a broader class of smooth-type certificates.

\section*{Acknowledgments}

I used LLMs to assist with computations, analysis, synthesis, and
verification/refinement in the preparation of this article, especially
GPT-5.6~Sol, GPT-5.$\{\text{2},\text{4},\text{5}\}$~Pro and Claude~Fable~5
(all accessed in part via Poe with the support of Quora, where I am
an advisor). In particular, after I had explored variants of the problem
through both LLM-assisted and analog methods over a period of months, an
exchange with GPT-5.6~Sol suggested changing the scale of the starting
point; that suggestion precipitated the present paper.  The problem,
final methods, and eventual written form are my own; and of course any
errors remain my responsibility.

I gratefully acknowledge helpful comments from Ben Golub and Ken Ono,
as well as Refine.ink, which provided transformative feedback on a
prior draft.  This work was conducted while I was visiting the
Technological Innovation, Entrepreneurship, and Strategic Management
(TIES) Group at the MIT Sloan School of Management; I greatly
appreciate their hospitality.

\providecommand{\bysame}{\leavevmode\hbox to3em{\hrulefill}\thinspace}
\providecommand{\MR}{\relax\ifhmode\unskip\space\fi MR }
\providecommand{\MRhref}[2]{%
  \href{http://www.ams.org/mathscinet-getitem?mr=#1}{#2}
}
\providecommand{\href}[2]{#2}

\end{document}